\documentclass{ws-jktr}

\begin{document}

\markboth{Louis H. Kauffman}
{}

\catchline{}{}{}{}{}

\title{Virtual Knot Cobordism and the Affine Index Polynomial}

\author{Louis H. Kauffman}

\address{Department of Mathematics, Statistics and Computer Science \\ 851 South Morgan Street   \\ University of Illinois at Chicago\\
Chicago, Illinois 60607-7045\\ and\\ Department of Mechanics and Mathematics\\ Novosibirsk State University\\Novosibirsk, Russia\\$<$kauffman@uic.edu$>$}

\maketitle

\begin{abstract}
 This paper studies cobordism and concordance for virtual knots. We define the affine index polynomial, prove that it is a concordance invariant for knots and links (explaining when it is defined for links), show that it is also invariant under certain forms of labeled cobordism and study a number of examples in relation to these phenomena. Information on determinations of the four-ball genus of some virtual knots is obtained by via the affine index polynomial in conjunction with results on the genus of positive virtual knots using joint work with Dye and Kaestner.
\end{abstract}



\section{Introduction}
This paper studies the concordance invariance of the affine index polynomial \cite{Affine}, denoted $P_{K}(t).$ This invariant is also called the {\it writhe polynomial}, 
$W_{K}(t),$ in the context of  Gauss diagrams, see \cite{Boden} where the $W$ notation is used and where a related polynomial is denoted by $P_{K}(t)$ (there should be no confusion). In this paper we work in the context of virtual knot and link diagrams and extend the definitions
in \cite{Affine} to an affine index polynomial for links (with affine labeling as described in the body of the paper). We prove that this generalized invariant is a concordance invariant of knots and links.\\

The paper is organized as follows.
In Section 2 we review basics of virtual knot theory. In Section 3 we review virtual link cobordism \cite{VKC} and discuss the four-ball genus of virtual links, recalling the construction of the virtual Seifert surface for a virtual link and the theorem \cite{DKK} that the four-ball genus of a positive virtual link is equal to the genus of its virtual Seifert surface. In Section 4 we re-develop the affine index polynomial \cite{Affine} extending its definition to labeled links and proving that it is a concordance invariant of virtual knots and links. The reader should note that concordance invariance of the affine index polynomial for knots is proved in \cite{Boden} by Gauss diagram techniques. In the present work we use affine labelings of the knot and link diagrams. In this context, we generalize the theorem to include links that are compatible with affine labeling. We also show how the affine index polynomial is invariant under certain cobordisms of knots and links that are compositions of saddle moves at crossings that have null weights in the affine labeling. This cobordism invariance was already observed in \cite{VKC} and here it is useful in understanding the core locus of 
the non-triviality of the polynomial. We illustrate this comment with a number of examples in the body of the paper.\\

\noindent {\bf Remark.} {\it Knotoids} are knot and link diagrams with free ends in possibly distinct regions, taken up to Reidemeister moves that never move an arc
across an endpoint. The subject of link cobordism and our results about the affine index polynomial in this paper all apply as well to the cobordism of knotoids. 
See \cite{GK1,GK2} for our work on knotoids. In these papers
we show how to define the affine index polynomial in the category of knotoids. In particular the theorems about concordance invariance of the affine index polynomial are valid in the knotoid category. This subject of cobordism of knotoids will be taken up in a paper distinct from the present work.\\

 \section{Virtual Knot Theory}
Virtual knot theory \cite{VKT,SVKT,DKT,DVK,Intro} studies a generalization of classical knot theory that we describe by using diagrams that include a virtual crossing that is neither over nor under. The virtual knot behind such a diagram can be regarded as an abstract knot diagram that is determined by the cyclic ordered structure of its crossing data. The virtual crossings are the result of immersing the abstract diagram into the plane. A diagrammatic theory generalizing the Reidemeister moves defines the virtual theory. Virtual knots can be studied by examining embeddings of curves in thickened surfaces of arbitrary
genus, up to the addition and removal of empty handles from the surface.  This paper, however, will concentrate on the diagrammatic point of view and will utilize the combinatorics of the virtual crossing structure. Classical knot theory embeds in virtual knot theory. The theory of virtual cobordism developed here is intended for the diagrammatic point of view. It will be the subject of other work by the author to forge relationships between the diagrammatic cobordisms described here and embedded cobordisms related to the surface embeddings for virtual diagrams.\\

In the diagrammatic theory of virtual knots one adds 
a {\em virtual crossing} (see Figure~\ref{Figure 1}) that is neither an over-crossing
nor an under-crossing.  A virtual crossing is represented by two crossing segments with a small circle
placed around the crossing point. \\

Moves on virtual diagrams generalize the Reidemeister moves for classical knot and link
diagrams.  See  Figure~\ref{Figure 1}.  Classical crossings interact with
one another according to the usual Reidemeister moves, while virtual crossings are artifacts of the structure in the plane. 
Adding the global detour move to the Reidemeister moves completes the description of moves on virtual diagrams. In  Figure~\ref{Figure 1} we illustrate a set of local
moves involving virtual crossings. The global detour move is
a consequence of  moves (B) and (C) in  Figure~\ref{Figure 1}. The detour move is illustrated in  Figure~\ref{Figure 2}.  Virtual knot and link diagrams that can be connected by a finite 
sequence of these moves are said to be {\it equivalent} or {\it virtually isotopic}. Figure~\ref{Figure 4} illustrates how a virtual knot can be interpreted in terms of the Gauss code (indicating a sequence of over and undercrossings with signs that determine the diagram) and via an embedded curve in a thickened surface.\\

Virtual knot diagrams are usually represented as diagrams in the plane, but the theory is not changed if one regards the diagram as drawn on the surface of a two dimensional sphere. Moves that swing an arc around the 
two-sphere can be accomplished in the plane by using the detour move. Again, we refer to the reference papers at the beginning of this section for the reader who is interested in more details about the foundations of 
virtual knot and link theory.\\

\begin{figure}[htb]
     \begin{center}
     \begin{tabular}{c}
     \includegraphics[width=8cm]{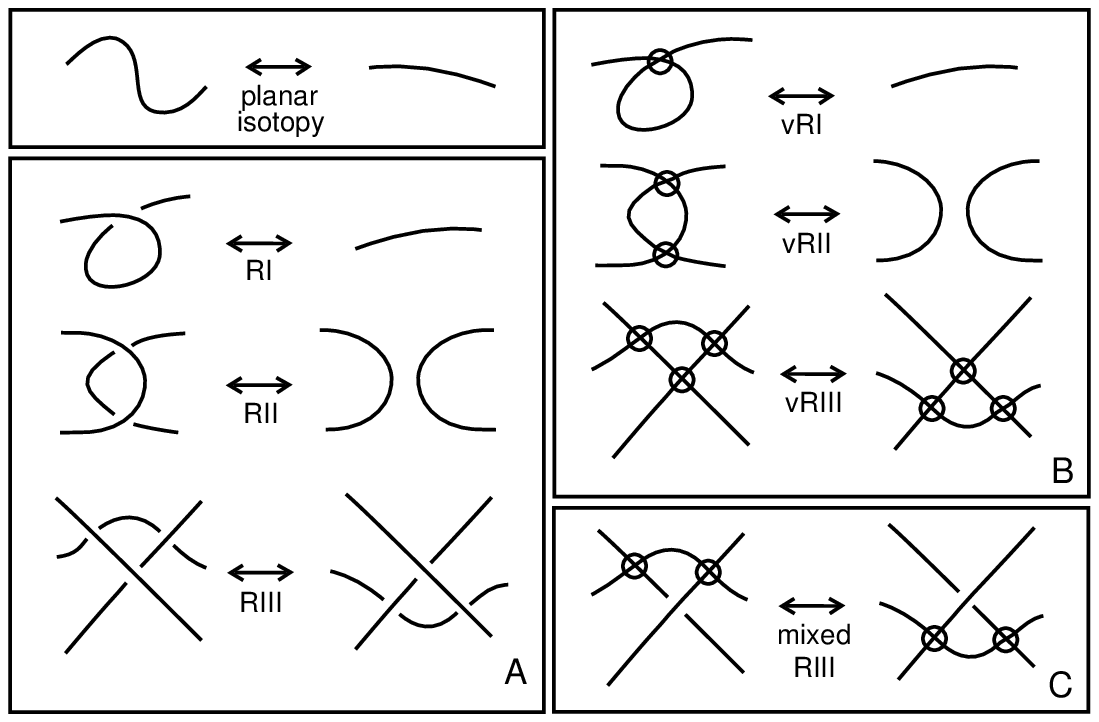}
     \end{tabular}
     \caption{\bf Moves}
     \label{Figure 1}
\end{center}
\end{figure}

\begin{figure}[htb]
     \begin{center}
     \begin{tabular}{c}
     \includegraphics[width=6cm]{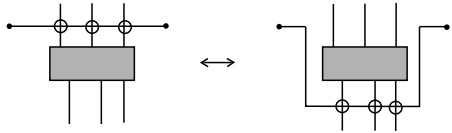}
     \end{tabular}
     \caption{\bf Detour Move}
     \label{Figure 2}
\end{center}
\end{figure}

\begin{figure}[htb]
     \begin{center}
     \begin{tabular}{c}
     \includegraphics[width=8cm]{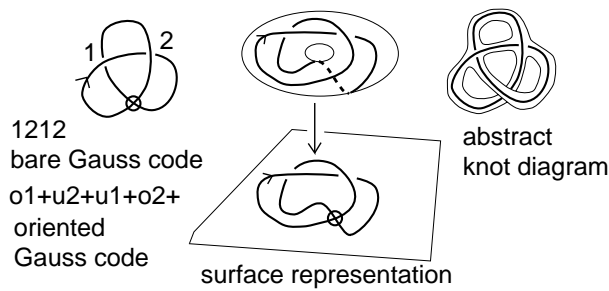}
     \end{tabular}
     \caption{\bf Representations of Virtual Knots}
     \label{Figure 4}
\end{center}
\end{figure}

One can understand virtual diagrams as representatives for oriented Gauss codes or Gauss diagrams \cite{GPV}, \cite{VKT,SVKT}. 
Such codes do not always have planar realizations. An attempt to embed such a code in the plane
leads to the production of the virtual crossings. The detour move of Figure~\ref{Figure 2} makes the particular choice of virtual crossings 
irrelevant. Virtual isotopy (generated by Reidemeister moves and detour moves) the same as the equivalence relation generated on the collection
of oriented Gauss codes by the analog of Reidemeister moves on these codes. 
(That is, one can translate the diagrammatic Reidemeister moves to combinatorial operations on the Gauss codes.) In Figure~\ref{Figure 4} we illustrate a number of representations of virtual knots, the Gauss code, the representation as a diagram on a 
closed surface, the projection to a virtual diagram from such a surface, and the abstract link diagram that can be regarded as a neighborhood of the embedding of a diagram in a closed surface. The interested reader will find
more details about these representations in the papers we have mentioned about virtual knot theory.\\

\noindent{\bf Remark.} In Figure~\ref{Figure 1} we have indicated the three Reidemeister moves by $RI, RII, RIII.$ We shall often refer to these moves as the first, second and third Reidemeister moves.\\

Many invariants for classical knots extend to invariants of virtual knots and links, including the bracket polynomial model for the Jones polynomial, Khovanov homology, the arrow polynomial extension of the bracket polynomial, fundamental group and quandles. We refer the reader to \cite{Intro,DKK} for more information about these invariants. In this paper we will concentrate on the affine index polynomial \cite{Affine} and its cobordism properties.\\

\subsection{Parity and Odd Writhe}
Parity is an important theme in virtual knot theory and figures in many investigations of this subject.
In a virtual knot diagram there can be both even and odd crossings.  A crossing is {\it odd} if 
it flanks an odd number of symbols in the Gauss code of the diagram. A crossing is {\it even} if 
it flanks an even number of symbols in the Gauss code of the diagram. For example, in 
Figure~\ref{Figure 4} we illustrate a virtual knot with bare Gauss code $1212.$ Both crossings in the diagram are odd. In any classical knot diagram all crossings are even.\\

In \cite{SL} we introduced the  {\it odd writhe} $J(K)$ for any virtual diagram $K.$ $J(K)$  is the sum of the signs of the odd crossings. Classical diagrams have zero odd writhe. Thus if $J(K)$ is non-zero, then $K$ is not equivalent to any classical knot. For the mirror image $K^{*}$ of any diagram $K,$ we have the formula $J(K^{*}) = - J(K).$ Thus, when $J(K) \ne 0,$  we know that the knot $K$ is
not classical and not equivalent to its mirror image. Parity does all the work in this simple invariant.
For example, if $K$ is the virtual knot in Figure~\ref{Figure 4}, the we have $J(K) = 2.$ Thus $K,$ the simplest virtual knot, is non-classical and it is chiral (inequivalent to its mirror image).\\

In Section 4 of this paper we will examine a generalization of the odd writhe to a polynomial invariant of virtual knots (the affine index polynomial \cite{Affine}) and we shall see how these invariants behave under 
cobordism and concordance of virtual knots, as described in Section 3.\\

\section{Virtual Knot Cobordism and Concordance}
\begin{definition} Two oriented knots or links $K$ and $K'$ are {\it virtually cobordant}  if one may be obtained from the other by a sequence of virtual isotopies (Reidemeister moves plus detour moves) plus births, deaths and oriented saddle points, as illustrated in  Figure~\ref{saddle}. A {\it birth} is the introduction into the diagram of an isolated unknotted circle. A {\it death} is the removal from the diagram
of an isolated unknotted circle. A saddle point move results from bringing oppositely oriented arcs 
into proximity and resmoothing the resulting site to obtain two new oppositely oriented arcs. See Figure~\ref{saddle} for an illustration of the process. Figure~\ref{saddle} also illustrates the {\it schema} of surfaces that are generated by  cobordism process. These are abstract surfaces with well defined genus in terms of the sequence of steps in the cobordism. In the Figure we illustrate two examples of genus zero, and one example of genus 1. We say that a cobordism has genus $g$ if its schema has genus $g.$ Two virtual knots or links are {\it virtually concordant} if there is a cobordism of genus zero  connecting them. Note that virtual concordance is a special case of virtual cobordism. We shall often just say {\it cobordant} or {\it concordant} with the word virtual assumed.\\

A virtual knot is said to be a {\it slice} knot if it is virtually concordant to the unknot, or equivalently if it is virtually concordant to the empty knot (The unknot is concordant to the empty knot via one death). As we shall see below, {\it every virtual knot or link is cobordant to the unknot}. Another way to say this, is to say that there is a {\it virtual surface} (schema) whose boundary is the given virtual knot. The reader should note that when we speak of a virtual surface, we mean a surface schema that is generated by saddle moves, maxima and minima as described above. \end{definition}

\noindent {\bf Remark.} The reader should note the sharp difference between the concepts of {\it cobordism} of virtual knots and {\it concordance} of virtual knots. Two knots that are cobordant can mutually bound a virtual 
surface of arbitrary genus. Two knots that are concordant must mutually bound a surface of genus zero. Just as in the classical case of knot concordance, this is a highly restricted relationship and one wants to be able to 
determine whether two knots are concordant, whereas any two knots are cobordant. On the other hand, the least genus for a cobordism surface between two knots or between a knot and the unknot is of great interest.\\

\begin{definition} The {\it four-ball genus} $g_{4}(K)$ of a virtual knot or link $K$ is the least genus among all virtual surfaces obtained by virtual cobordism that bound $K.$ As we shall see below, there is a simple upper bound on the four-ball genus for any virtual knot or link and a definite result for the four-ball genus of positive virtual knots \cite{DKK}. Note that in this definition of four-ball genus we have not made reference to an embedding of the surface in the four-ball $D^{4}.$ The surface constructed by a virtual cobordism is, for this paper, an abstract surface with a well-defined genus. This same surface can be given the structure of virtual surface diagram analogous to a virtual knot or link diagram (see \cite{Takeda}) but we will not discuss this aspect of virtual surfaces in the present paper. Note that virtual slice knots are virtual knots $K$ with 
$g_{4}(K) = 0.$\end{definition}

\begin{figure}
     \begin{center}
     \begin{tabular}{c}
     \includegraphics[width=8cm]{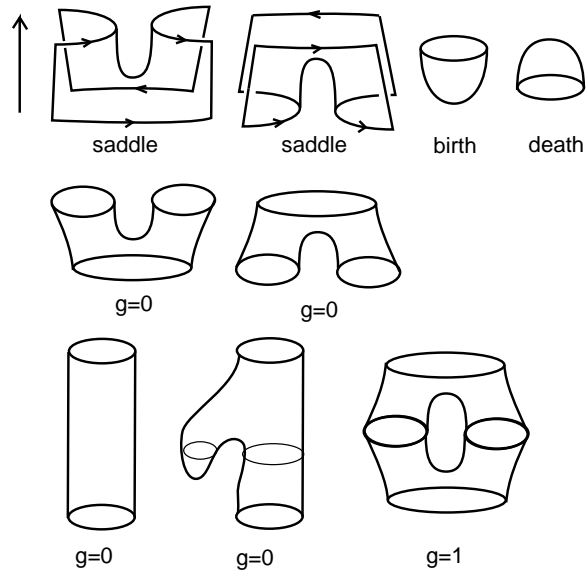}
     \end{tabular}
     \caption{\bf Saddles, Births and Deaths}
     \label{saddle}
\end{center}
\end{figure}

In Figure~\ref{vstevedore} we illustrate the {\it virtual stevedore's knot} that we will denote by $VS,$ and show that it is a slice knot in the sense of the above definition. This figure illustrates how the surface schema whose boundary in the virtual stevedore is evolved via the saddle point that produces two virtually unlinked curves that are isotopic to a pair of curves that can undergo deaths to produce the genus zero slicing surface. We will use this example to illustrate our theory of virtual knot cobordism, and the questions that we are investigating.\\

In Figure~\ref{vertical} we illustrate a connected sum of a virtual knot $K$ and its {\it vertical mirror image} $K^{!}.$ The vertical mirror image is obtained by reflecting the diagram in a plane perpendicular to the plane 
of the diagram {\it and reversing the orientation of the resulting diagram}. We indicate this particular connected sum by $K \sharp K^{!}.$ While connected sum of virtual knots is not in general defined except by a diagrammatic choice, we do have a diagrammatic definition of this 
connected sum and it is the case that {\it $K \sharp K^{!}$ is a slice knot for any virtual diagram $K.$} The idea behind the proof of this statement is illustrated in Figure~\ref{verticalslice}. Saddle points can be made by pairing arcs across the mirror and the diagram resolves into a collection of virtual trivial circles. We omit the detailed proof of this fact about virtual concordance. This result is a direct generalization of the 
corresponding result for classical knots and links \cite{FoxMilnor}.\\

\begin{figure}
     \begin{center}
     \begin{tabular}{c}
     \includegraphics[width=6cm]{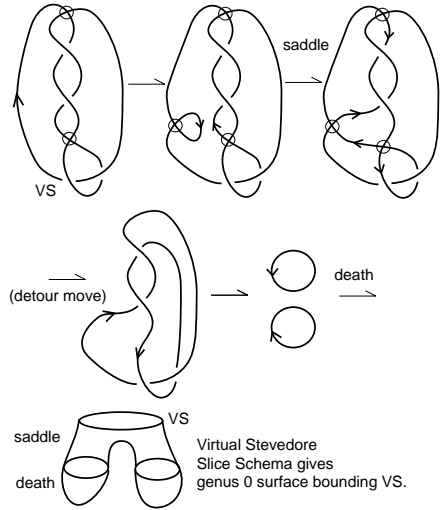}
     \end{tabular}
     \caption{\bf Virtual Stevedore is Slice}
     \label{vstevedore}
\end{center}
\end{figure}

\begin{figure}
     \begin{center}
     \includegraphics[width=6cm]{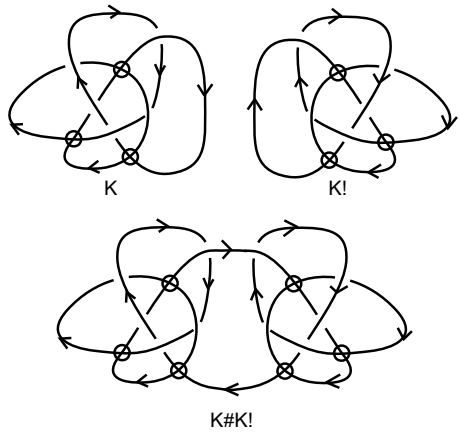}
     \caption{\bf Vertical Mirror Image}
     \label{vertical}
\end{center}
\end{figure}

\begin{figure}
     \begin{center}
     \includegraphics[width=6cm]{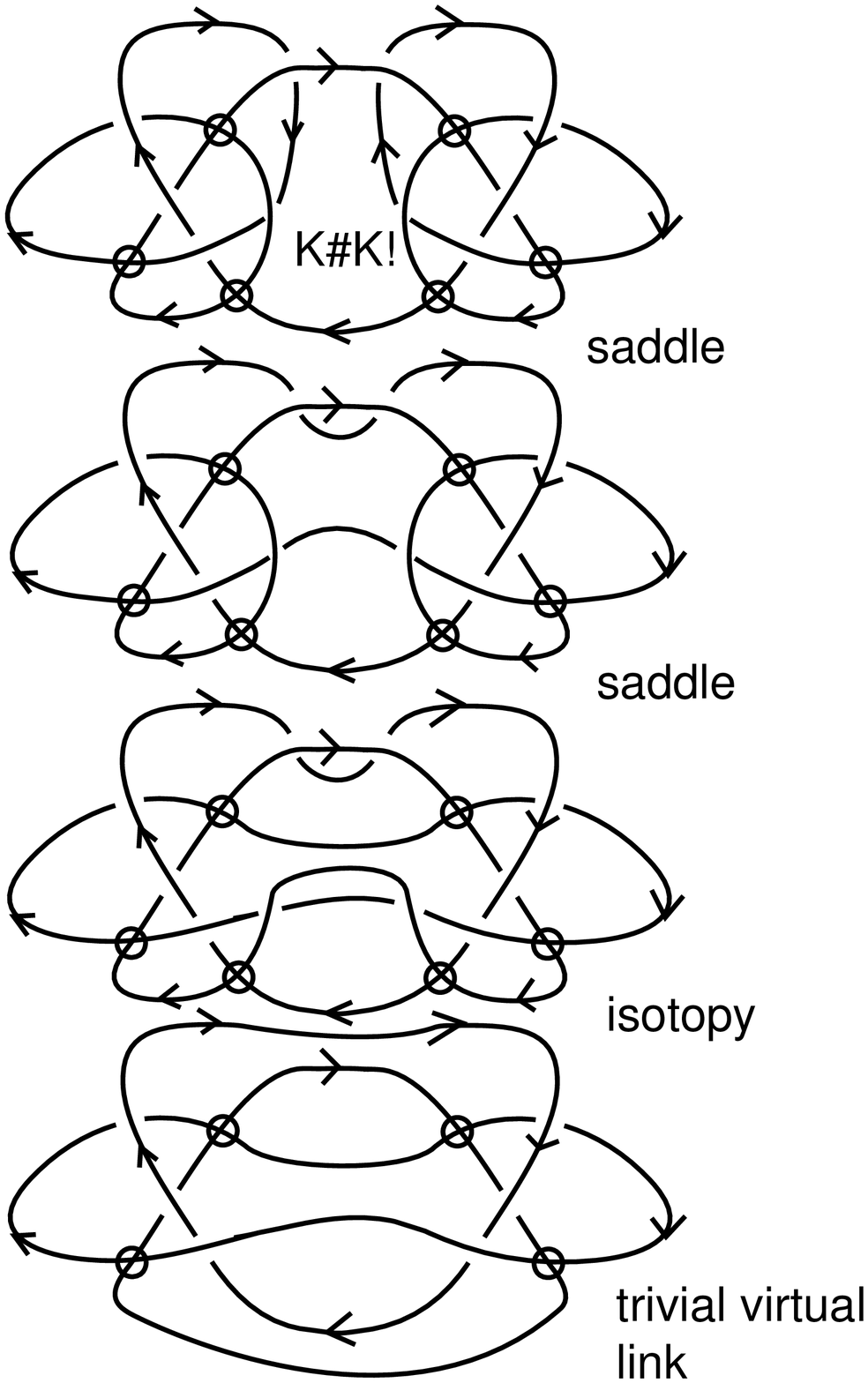}
     \caption{\bf Connected Sum With Vertical Mirror Image is Slice}
     \label{verticalslice}
\end{center}
\end{figure}

\subsection{Spanning Surfaces for Knots and Virtual Knots and the Four-Ball Genus of Positive Virtual Knots}
Every oriented classical knot or link bounds an embedded orientable surface in three-space. A representative surface of this kind can be obtained by Seifert's algorithm (See \cite{OK,FKT,KP}). We illustrate Seifert's algorithm for a trefoil diagram in Figure~\ref{seifert}. The algorithm proceeds as follows: At each oriented crossing in a given diagram 
$K,$ smooth that crossing in the oriented manner (reconnecting the arcs locally so that the crossing disappears and the connections respect the orientation). The result operation is a collection of oriented simple closed curves in the plane, usually called the {\it Seifert circles}. To form the {\it Seifert surface} $F(K)$ for the diagram $K,$ attach disjoint discs to each of the Seifert circles, and connect these discs to one another by local half-twisted bands at the sites of the smoothing of the diagram. This process is indicated in Figure~\ref{seifert}. In that figure we have not completed the illustration of the outer disc.\\

\begin{lemma} Let $K$ be a classical knot diagram with $n$ crossings and $r$ Seifert circles.
Then the genus of the Seifert Surface $F(K)$ is given by the formula
$$g(F(K)) =(1/2)( -r + n +1).$$
\end{lemma}

\begin{proof} See \cite{VKC}. \end{proof}

For any classical knot $K,$ there is a surface bounding that knot in the four-ball that is homeomorphic to the Seifert surface.  One constructs this surface by  pushing the Seifert
surface into the four-ball keeping it fixed along the boundary. A different description of
this surface as indicated in Figure~\ref{classicalcob}. We {\it perform a saddle point transformation at every crossing of the diagram.} The result is a collection of unknotted and unlinked curves.  We then bound each of these curves by discs (via deaths of circles) and obtain a surface 
$S(K)$ embedded in the four-ball with boundary $K.$ As the reader can easily see, the curves produced by the saddle transformations are in one-to-one correspondence with the Seifert circles for
$K$ and $S(K)$ is homeomorphic with the Seifert surface $F(K).$
Thus $g(S(K)) =(1/2)( -r + n +1).$  \\

We generalize the Seifert surface to a surface $S(K)$ for virtual knots $K$ by performing exactly these saddle moves at each classical crossing of the virtual knot.
View Figure~\ref{virtseifert} and Figure~\ref{vstevedoreseifert}.  The result is a collection of unknotted curves that are isotopic (by the first classical Reidemeister move) to curves with only virtual crossings. Once the first Reidemeister moves are performed, these curves are identical with the 
{\it virtual Seifert circles} obtained from the diagram $K$ by smoothing all of its classical crossings.
We can isotope these circles into a disjoint collection of circles and cap them with discs in the four-ball. The result is a virtual surface $S(K)$ whose boundary is the given virtual knot $K.$ We will use the terminology {\it virtual surface in the four-ball}
for this surface schema. In the case of a virtual slice knot, the knot bounds a virtual surface of genus zero. We have the following lemma.\\

\begin{lemma} Let $K$ be a virtual knot, then the virtual Seifert surface $S(K)$ constructed above
has genus given by the formula $$g(S(K)) = (1/2)(-r + n + 1)$$ where $r$ is the number of virtual Seifert
circles in the diagram $K$ and $n$ is the number of classical crossings in the diagram $K.$ \end{lemma}

\begin{proof} See \cite{VKC}. \end{proof}

\noindent{\bf Remark.} Note that it follows from the above discussion that if a diagram $K'$ is obtained from a diagram $K$ by replacing a crossing in $K$ by its oriented smoothing, then $K'$ is cobordant to $K$ via a single 
saddle point move. We will use this observation repeatedly in the rest of the paper.\\

\noindent{\bf Remark.} For the virtual stevedore in Figure~\ref{vstevedoreseifert}
 there is a lower genus surface (genus zero as we have already seen in Section 2) than can be produced by cobordism using the virtual Seifert surface. In that same figure we have illustrated a diagram $D$ with the same projected diagram as the virtual stevedore, but $D$ has all positive crossings. In this case we can prove that there is no virtual surface for this diagram  $D$ of four-ball genus less than $1.$ 
In fact, we have the following result. This theorem is a generalization of a corresponding result for classical knots due to Rasmussen \cite{Ras}.\\

\begin{theorem} [On Four-Ball Genus for Positive Virtual Knots  \cite{DKK}] Let $K$ be a positive virtual knot (i.e. all classical crossings in $K$ are positive), then the four-ball genus $g_{4}(K)$ is given by the formula
$$g_{4}(K) = (1/2)(-r + n + 1) = g(S(K))$$ where $r$ is the number of virtual Seifert circles in the diagram
$K$ and $n$ is the number of classical crossings in this diagram. In other words, the virtual Seifert surface for $K$ represents its minimal four-ball genus.\end{theorem}

\noindent{\bf Remark.} This theorem is proved by using a generalization of integral Khovanov homology to virtual knot theory originally devised by Manturov \cite{MBook}. In \cite{DKK} we reformulate this theory and show that it generalizes to the Lee homology theory (a variant of Khovanov homology) as well. With this theorem, we know the genus for an infinite class of virtual knots and can begin the deeper exploration of genus for non-positive virtual knots and links.\\

\begin{figure}
     \begin{center}
     \begin{tabular}{c}
     \includegraphics[width=5cm]{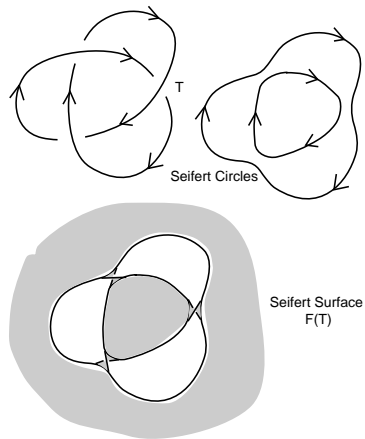}
     \end{tabular}
     \caption{\bf Classical Seifert Surface}
     \label{seifert}
\end{center}
\end{figure}

\begin{figure}
     \begin{center}
     \begin{tabular}{c}
     \includegraphics[width=5cm]{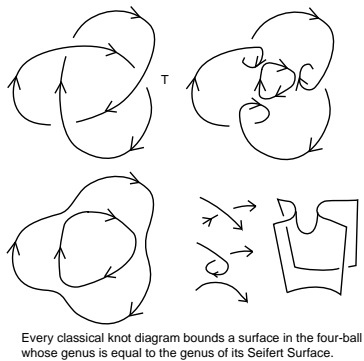}
     \end{tabular}
     \caption{\bf Classical Cobordism Surface}
     \label{classicalcob}
\end{center}
\end{figure}

\begin{figure}
     \begin{center}
     \begin{tabular}{c}
     \includegraphics[width=5cm]{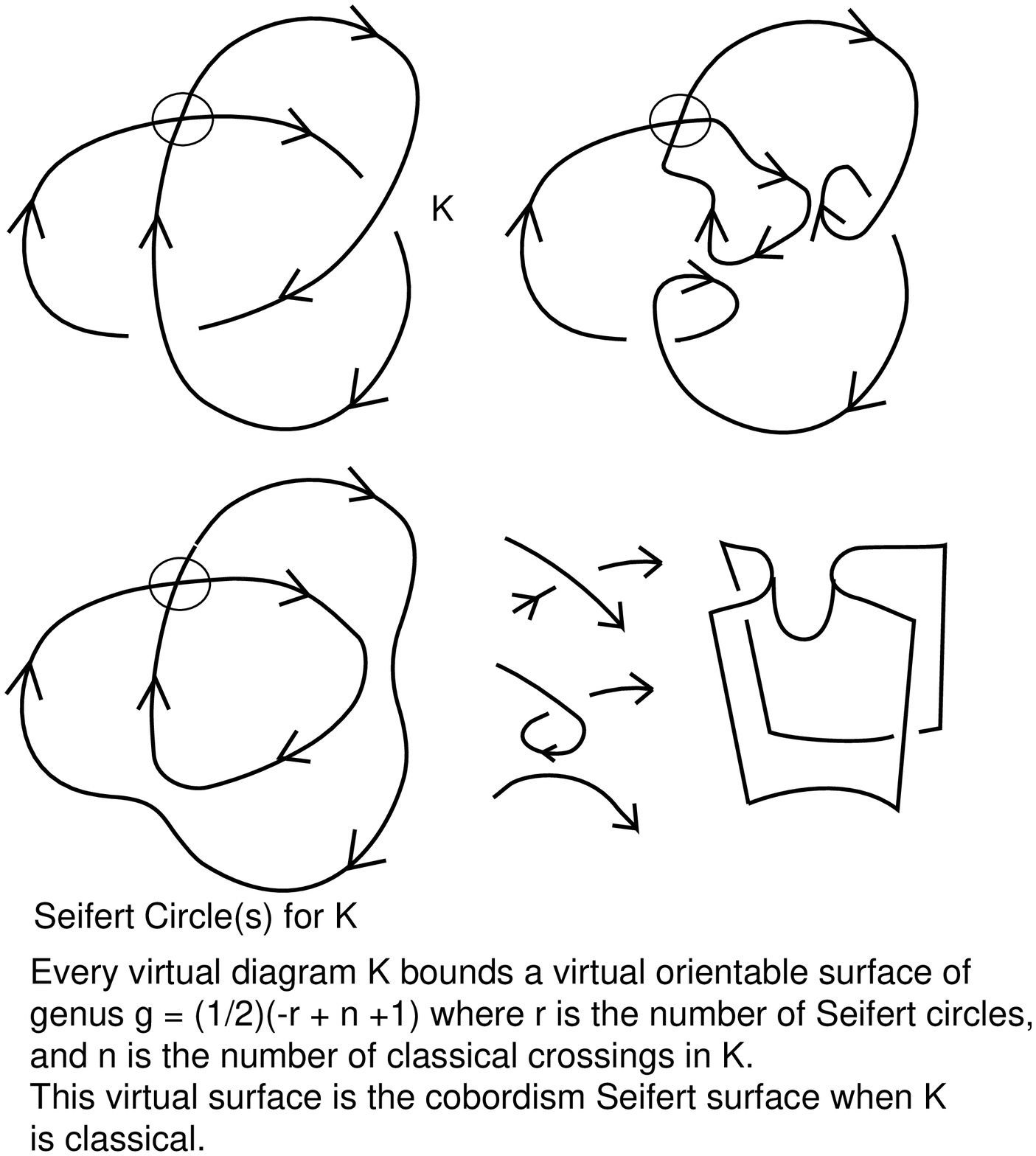}
     \end{tabular}
     \caption{\bf Virtual Cobordism Seifert Surface}
     \label{virtseifert}
\end{center}
\end{figure}

\begin{figure}
     \begin{center}
     \begin{tabular}{c}
     \includegraphics[width=5cm]{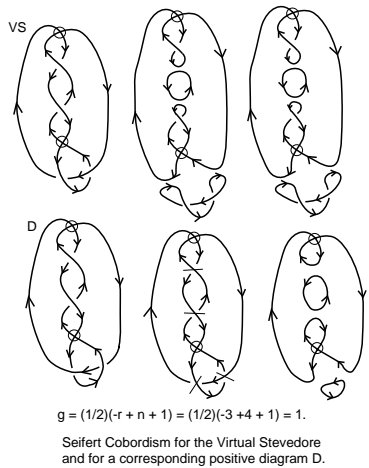}
     \end{tabular}
     \caption{\bf Virtual Stevedore Cobordism Seifert Surface}
     \label{vstevedoreseifert}
\end{center}
\end{figure}

\section {The Affine Index Polynomial Invariant}
The purpose of this section is to show that the affine index polynomial invariant \cite{Affine} of virtual knots is a concordance invariant (see Definition 3.1), and to extend this invariant and its
properties to virtual links. To this purpose, we begin by reviewing the definition of the affine index polynomial and recall its basic properties. We use the diagrammatic point of view in this section and do not
use Gauss codes for the definitions and constructions.\\

We first describe how to calculate the affine index polynomial, then prove invariance under virtual link equivalence, and then prove concordance invariance.
Calculation begins with a flat oriented virtual knot diagram  (the classical crossings in a flat diagram do not have choices made for over or under). An {\it arc} of a flat diagram is an edge of the $4$-regular graph that represents the diagram. An edge extends from one classical crossing to the next in orientation order. An arc may have many virtual crossings, but it begins at a classical crossing and ends at another classical crossing. We label each arc $c$ in the diagram with an integer $\lambda(c)$ so that an arc that meets  a classical crossing and crosses to the left increases the label by one, while an arc that meets a classical crossing and crosses to the right decreases the label by one. See Figure~\ref{example1} for an illustration of this rule. Such integer
labeling can always be done for any virtual or classical link diagram \cite{Affine}. In a virtual diagram the labeling is unchanged at a virtual crossing, as indicated in Figure~\ref{example1}. One can start by choosing some arc to have an arbitrary integer label, and then proceed along the diagram labeling all the arcs via this crossing rule.
We call such an integer labeling of a diagram an {\it affine labeling} of the diagram and sometimes just a {\it labeling} of the diagram. In \cite{Affine} we use the equivalent term {\it Cheng labeling} for the affine labeling.\\

\begin{figure}
     \begin{center}
     \includegraphics[width=6cm]{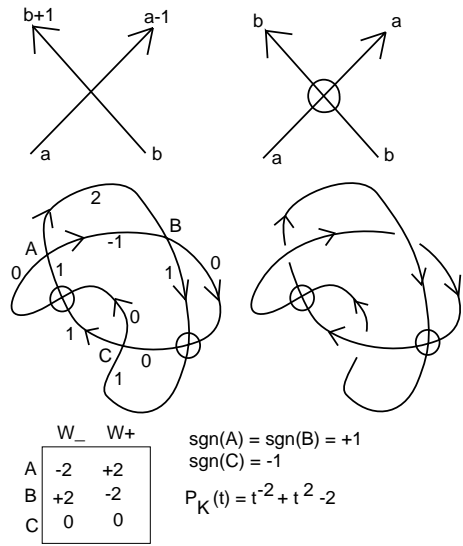}
     \caption{\bf Labeled Flat Crossing and an Example}
     \label{example1}
\end{center}
\end{figure}

\begin{figure}
     \begin{center}
     \includegraphics[width=5cm]{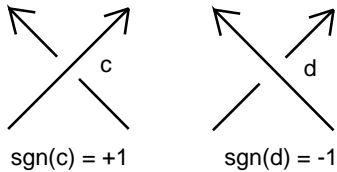}
     \caption{\bf Crossing Signs}
     \label{crossingsign}
\end{center}
\end{figure}

\noindent {\bf Remark.}  We discuss the algebraic background to this invariant in \cite{Affine}.
Once we have a labeled flat diagram, we assign two {\it weights} , $W_{+}$ and $W_{-}$ to each of its crossings according to the definition below.
Then given a diagram with classical crossings $j$ we assign a weight $W(j)$ to be $W_{+}$ if $c$ is a positive classical crossing, and $W_{-}$ if $j$ is a negative classical crossing.\\

\begin{definition} Given a labeled flat diagram we define two numbers at each classical crossing: $W_{-}$ and 
$W_{+}$ as shown in Figure~\ref{example1}. If we have a labeled classical crossing with left incoming arc $a$ and 
right incoming arc $b$ then the right outgoing arc is labeled $d= a - 1$ and the left outgoing arc is labeled
$c= b+1$ as shown in Figure~\ref{example1}. We then define
$W_{+} = a -( b +1)$ and $W_{-} = b - (a -1).$
Note that $W_{-} = - W_{+}$ in all cases.\end{definition}

\begin{definition} Given a crossing $c$ in a diagram $K,$ we let $sgn(c)$ denote the sign of the crossing.
The sign of the crossing is plus or minus one according to the convention shown in 
Figure~\ref{crossingsign}.
The {\it writhe}, $wr(K),$  of the diagram $K$ is the sum of the signs of all its crossings.
For a virtual link diagram, labeled in the integers according to the scheme above, and a crossing 
$c$ in the diagram, define {\it the weight of the crossing} $W_{K}(c)$ by the equation
$$W_{K}(c) = W_{sgn(c)}(c)$$ where $W_{sgn(c)}(c)$ refers to the underlying flat diagram for $K$. Thus $W_{K}(c)$ is $W_{\pm}(c)$ 
according as the sign of the crossing is plus or minus. {\it We shall often indicate the weight of a crossing $c$ in a knot diagram $K$ by $W(c)$ rather than $W_{K}(c).$}\end{definition}

\noindent{\bf Remark.} Note that in Figure~\ref{example1} we have flat crossings $A,B,C$ and corresponding crossings in the virtual knot $K.$ The Figure illustrates that 
$W_{K} (A) = -2, W_{K}(B) = +2, W_{K}(C) = 0.$\\

\begin{definition} Let $K$ be a virtual knot diagram. Define the {\it Affine Index Polynomial of $K$} by the equation
 $$P_{K} = \sum_{c} sgn(c)(t^{W_{K}(c)} - 1) = \sum_{c} sgn(c)t^{W_{K}(c)} - wr(K)$$ where the summation is over all classical crossings in the virtual knot diagram $K.$
The Laurent polynomial $P_{K}$
is an invariant of virtual knots, as we shall recall below, and we shall show that it is a concordance invariant.
Note that we can rewrite this definition as follows:
$$P_{K} = \sum_{n=1}^{\infty} wr_{n}(K) (t^{n} -1) $$
where $$wr_{n}(K) = \sum_{c: W_{K}(c) = n}sgn(c).$$ \end{definition}

We can think of these numbers $wr_{n}(K)$ as special writhes for the virtual knot diagram, similar in spirit to the odd writhe. Each $wr_{n}(K)$ for $n=1,2,\cdots$ is an invariant of the virtual knot $K.$
Note also that a crossing $c$ in $K$ is odd (by our previous definition) if and only if $W_{K}(c)$ is odd. Thus, if $J(K)$ denotes the odd writhe of $K,$ then 
$$J(K) = \sum_{c: W_{K}(c) \, odd}sgn(c) = \sum_{n \, odd} wr_{n}(K).$$

\noindent{\bf Remark.} We define the {\it Flat Affine Index Polynomial}, $PF_{K}$, for a flat virtual knot $K$ (in  a flat virtual link the classical crossings are immersion crossings, neither over not
under, Reidemeister moves are allowed independent of over and under, but virtual crossings still take detour precedence over classical crossings \cite{VKT}) by the formula
$$PF_{K}(t) =\sum_{c} (t^{|W_{K}(c)|} + 1)$$ where the polynomial is taken over the integers modulo two, but the exponents (the absolute values of the weights at the crossings) are integral.
It is not hard to see that $PF_{K}(t)$ is an invariant of flat virtual knots, and that the concordance results of the present paper hold in the flat category for this invariant. These results will be a subject of a separate paper.\\
 
\noindent {\bf Remark.} In Figure~\ref{example1} we show the computation of the weights for a given flat diagram and the computation of the polynomial for a virtual knot $K$ with this underlying diagram. The knot $K$ is an example of a virtual knot with unit Jones polynomial. The polynomial $P_{K}$ for this knot has the value
$$P_{K} = t^{-2} + t^{2} - 2,$$ showing that this knot is not isotopic to a classical knot. \\

\subsection{Invariance of $P_{K}(t).$}
In order to show the invariance and well-definedness of  $P_{K}(t)$ we must first show the existence of affine labelings of flat virtual knot diagrams. We do this by showing that any virtual knot diagram $K$ that overlies a given flat diagram $D$ can be so labeled. \\

\begin{proposition} Any flat virtual knot diagram has an affine labeling. \end{proposition}

\begin{proof} This proposition is proved in \cite{Affine}. The main point is that on traversing the entire diagram, one goes through each crossing twice. The combination of these two operations results in a total change of zero. Hence, whatever label one begins with, the return label after a complete circuit of a diagram component will be the same as the start label. \end{proof}

\begin{definition} Not all multi-component virtual diagrams can be labeled. See Figure~\ref{nolabel} for such an example. We call a multi-component diagram $D$ {\it compatible} if every component of the diagram
has algebraic intersection number zero (taking signed intersection numbers in the plane) with the other components in $D.$ \end{definition}

\noindent We observe the following

\begin{lemma} Let $D$ be a multi-component virtual diagram. Then $D$ can be given an affine labeling if and only if it is compatible. \end{lemma}

\begin{proof} In any traverse of a given component of $D$ one will meet  external crossings each once, and increment or decrement 
the labeling according as the crossing has positive or negative sign with respect to this component. Self-crossings are met twice, once as an increment and once as a decrement.
Thus the total traverse will not change the initial label if and only if the algebraic intersection number of the given component with the rest of the diagram is zero. Since this must hold for each component of the diagram
$D$, we conclude that $D$ can be labeled if and only if $D$ is compatible. \end{proof}

\begin{figure}
     \begin{center}
     \includegraphics[width=5cm]{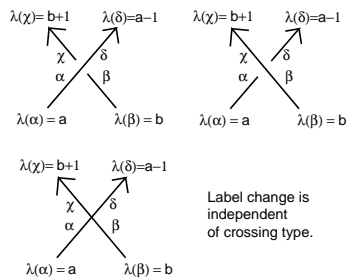}
     \caption{\bf Labels for Crossings}
     \label{lambda}
\end{center}
\end{figure}

\begin{figure}
     \begin{center}
     \includegraphics[width=3cm]{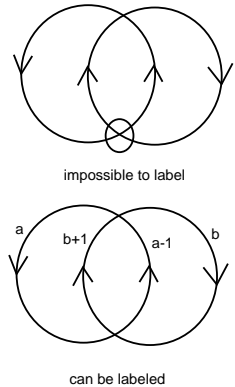}
     \caption{\bf Possible and Impossible Labels for Links}
     \label{nolabel}
\end{center}
\end{figure}

\noindent{\bf Remark.} If we follow the algorithm described in Figure~\ref{example1} to compute a labeling, using a different starting value, the resulting labeling will differ from the first labeling by a constant integer at every label. Since the polynomial is defined in terms of the differences $W_{\pm}(c)$ at each classical crossing $c$ of $K,$ it follows that the weights $W_{\pm}$ as described above are well-defined. We can now state a result about the weights. See \cite{Affine} for the proof.
Let $\bar{K}$ denote the diagram obtained by reversing the orientation of $K$ and let $K^{*}$ denote the diagram obtained by switching all the crossings of $K.$  $\bar{K}$ is called the {\it reverse} of $K,$ and $K^{*}$ is called the {\it flat mirror image} of $K.$ We let $K^{!}$ denote the {\it vertical mirror image of $K$} as shown in Figure~\ref{vertical}.\\

\noindent The following proposition and its proof will be mostly found in \cite{Affine} except for the statements about the vertical mirror image $K^{!}.$ These statements are easily seen from the discussion here and so 
we do not give a proof of this proposition here.

\begin{proposition} Let $K$ be a virtual knot diagram and $W_{\pm}(c)$ the crossing weights
as given in Definitions 4.1, 4.2 and 4.3. If $\alpha$  is an arc of $K,$ let $\bar{\alpha}$ denote the
corresponding arc of $\bar{K}$, the result of reversing the orientation of $K.$ 
\begin{enumerate}
\item Let $c$ be a crossing of $K$ and let $\bar{c}$
denote the corresponding crossing of $\bar{K}$, then $W(\bar{c}) = - W(c).$ 
Hence, 
$$P_{\bar{K}}(t) = P_{K}(t^{-1}).$$ Similarly, for the flat mirror image we have $$P_{K^{*}}(t) = -P_{K}(t^{-1}),$$
and for the vertical mirror image $$P_{K^{!}}(t) = -P_{K}(t).$$
Thus this invariant changes $t$ to $t^{-1}$ when the orientation of the knot is reversed, and it 
changes global sign and $t$ to $t^{-1}$ when the knot is replaced by its flat mirror image.
\item If $K$ is a classical knot diagram, then for each crossing $c$ in $K$, 
$W(c) = 0$ and  $P_{K}(t) = 0.$
\item If $K \sharp L$ denotes a connected sum (the diagrams are joined by removing an arc from each, and connecting them) of $K$ and $L,$ then 
$$P_{K \sharp L} = P_{K} + P _{L}.$$
Thus, if $K \sharp K^{!}$ denotes a connected sum of a virtual knot with its vertical mirror image (see Figure~\ref{vertical}), then it follows from the above that 
$$P_{K \sharp K^{!}} = P_{K} - P_{K^{!}} = 0.$$
\end{enumerate}
\end{proposition}

We will now state the invariance of $P_{K}(t)$ under virtual isotopy. The reader will recall that
virtual isotopy consists in the classical Reidemeister moves plus virtual moves that are all generated by one generic detour move. The (unoriented) virtual isotopy moves are illustrated in Figure~\ref{Figure 1} and Figure~\ref{Figure 2}. In Figure~\ref{r12} and Figure~\ref{r3} we show the relevant information for verifying that  $P_{K}(t)$ is an invariant of oriented virtual isotopy. The reader can find the details of this proof for virtual knots in \cite{Affine}.\\

\begin{theorem} Let $K$ be a virtual knot diagram. Then the polynomial $P_{K}(t)$
is invariant under oriented virtual isotopy and is hence an invariant of virtual knots.
\end{theorem}

\begin{proof} See \cite{Affine}. \end{proof}

\begin{figure}
     \begin{center}
     \includegraphics[width=5cm]{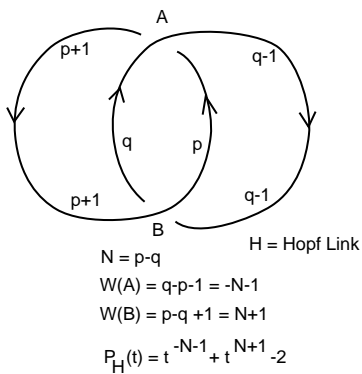}
     \caption{\bf Invariant for the Hopf Link}
     \label{hlink}
\end{center}
\end{figure}

\begin{figure}
     \begin{center}
     \includegraphics[width=5cm]{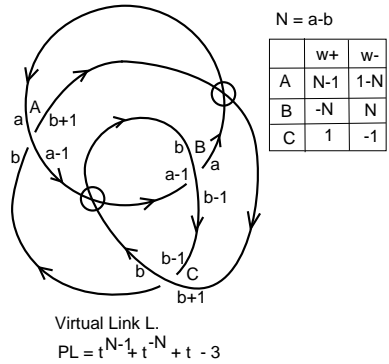}
     \caption{\bf Affine Index Invariant of a Virtual Link}
     \label{vlink}
\end{center}
\end{figure}

\begin{figure}
     \begin{center}
     \includegraphics[width=5cm]{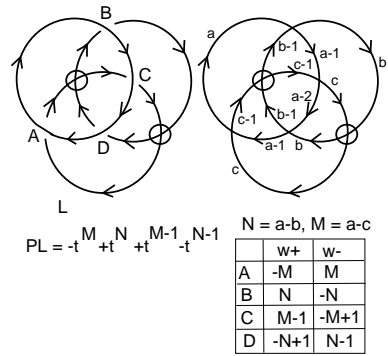}
     \caption{\bf Affine Index Invariant of a Virtual Borromean Rings}
     \label{vboro}
\end{center}
\end{figure}

\begin{figure}
     \begin{center}
     \includegraphics[width=5cm]{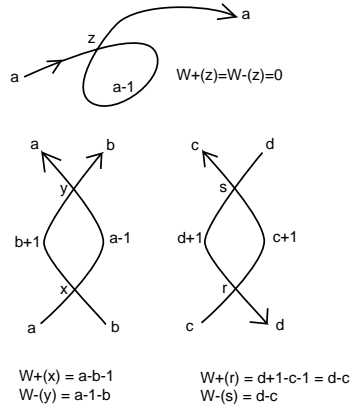}
     \caption{\bf Reidemeister Moves II and II}
     \label{r12}
\end{center}
\end{figure}

\begin{figure}
     \begin{center}
     \includegraphics[width=5cm]{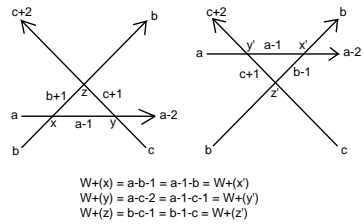}
     \caption{\bf Reidemeister Move III}
     \label{r3}
\end{center}
\end{figure}

\begin{figure}
     \begin{center}
     \begin{tabular}{c}
     \includegraphics[width=6cm]{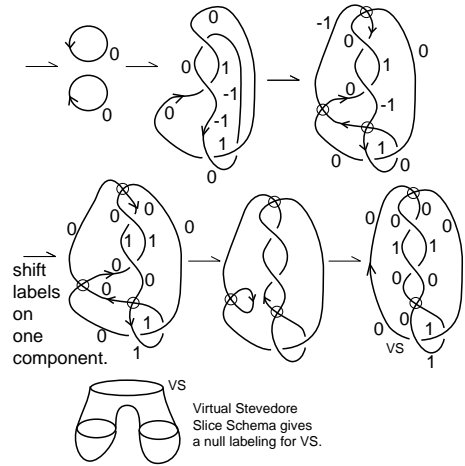}
     \end{tabular}
     \caption{\bf Virtual Stevedore has a Null Labeling}
     \label{labelvstevedore}
\end{center}
\end{figure}

\begin{figure}
     \begin{center}
     \begin{tabular}{c}
     \includegraphics[width=5cm]{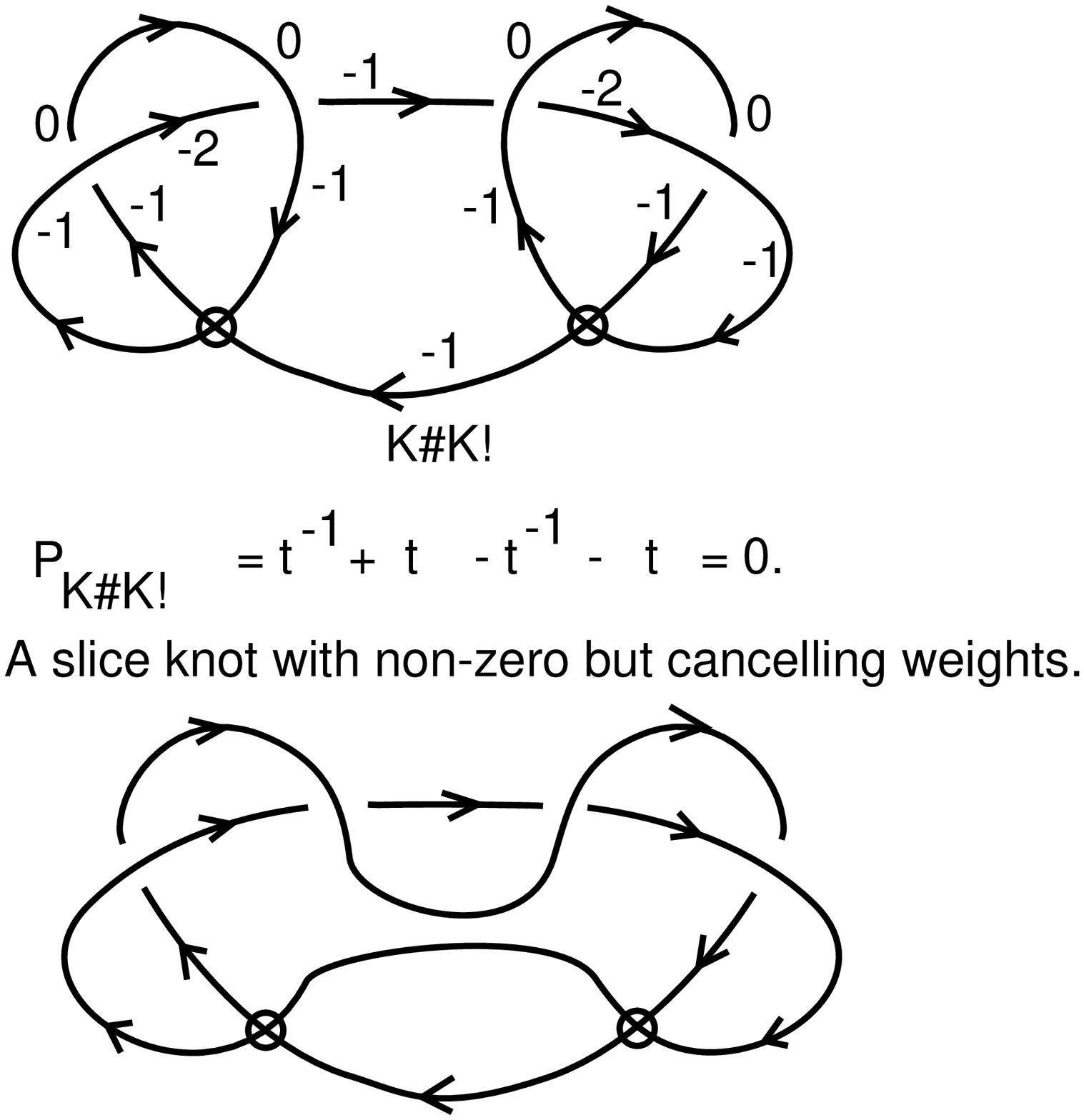}
     \end{tabular}
     \caption{\bf A Virtual Slice Knot with Non-Zero but Canceling Weights}
     \label{consum}
\end{center}
\end{figure}

\noindent {\bf Generalization of the Affine Index Polynomial from Knots to Links.} We are now in a position to generalize the invariant $P_{K}(t)$ to  cases of virtual and classical link diagrams. 
Some of the material in this discussion can be found in embryonic form in \cite{Affine}. Special
link diagrams can be affine colored according to our rules. For example, view Figure~\ref{hlink} to see a labeling of the classical Hopf link. Before analyzing this figure, consider the proof for 
the invariance of the polynomial $P_{K}(t).$ Affine coloring is uniquely inherited under Reidemeister
moves and the weights at the three crossings of the third Reidemeister move are permuted under the 
move. See Figure~\ref{r12} and Figure~\ref{r3}. These properties are true for the polynomial that we write for any affine-colored link.
Thus we can conclude that {\it if we are given pair $(L,C)$ where $L$ is a link diagram and $C$ is an 
affine-coloring of this diagram, then the polynomial $P_{L}(t)$, defined just as before, is an invariant
of the pair $(L,C)$} where a Reidemeister move applied to $(L,C)$ produces $(L',C')$ where $L'$ is the 
diagram obtained from $L$ by the move, and $C'$ is the coloring uniquely obtained from $C$ by the move. The resulting polynomial is an invariant of the link itself.
\\

Now go back to Figure~\ref{hlink} and note that we have given arbitrary labels $p$ and $q$ to arcs on the two components and obtained weights of the form 
$-N -1$ and $-N+1$ where $N= p-q.$ If we regard $N$ as an integer variable in the polynomial $P_{H} = t^{-N-1} + t^{-N + 1} -2$ (H is a positive Hopf link), then
this polynomial is an invariant of the link. This can be verified by applying Reidemeister moves to the link and showing that the value of $N$ is preserved.\\

In working with the invariant for  a link
we choose an {\it algebraic} starting value for each component of the link, using a different algebraic symbol for each component. It is convenient 
in displaying the weights to use new variables corresponding to the differences between algebraic labels. Thus in Figure~\ref{vlink}  we have a two component virtual link with labels $a$ and $b$ for each 
component and we define $N = a-b.$  At a crossing between two components the weights will be expressed uniquely in terms of $N$ (the difference between their algebraic labels). The invariant polynomial for the link has algebraic exponents involving these differences. In Figure~\ref{vlink} the polynomial is 
$P_{K} = t ^{N-1} + t^{-N} + t -3.$ \\

In Figure~\ref{vboro} we illustrate a link $L$ that is a virtual Borromean rings. No two components are linked but the triple is linked. The algebraically weighted affine index polynomial detects the linkedness
of these rings. Note that in this case we have two algebraic exponents $N$ and $M.$ We leave it to the reader to examine Figure~\ref{vboro} for more details about this example.\\

We have defined {\it compatibility}  (Definition 4.4) of multi-component diagrams above and proved that a multi-component diagram can be affine labeled if and only if it is compatible (Lemma 4.5).
Therefore compatible links have affine index polynomials. Just as we have remarked, such polynomials will in general have exponents that are new variables and that can be specialized to 
polynomials of labeled pairs. It is useful to have both the absolute link invariants and the labeled pair invariants. We shall use both types of invariant in the discussion to follow.\\

\subsection{Concordance Invariance of the Affine Index Polynomial}
The main result of this section is the \\

\begin{theorem} The affine index polynomial $P_{K}(t)$ is a concordance invariant of virtual knots $K$ and compatible virtual links (the links for which the invariant is defined). In the case of links we use 
integral affine labelings for the link, just as in the case of knots. For links, the genus zero concordance is restricted to one where all critical points can be paired in canceling maxima and saddles and canceling saddles and
minima. Note that this condition is automatically satisfied in the case of concordance of knots.
\end{theorem}

\begin{proof} Suppose that $K$ is concordant to $K'.$ Then there is a genus zero sequence of births deaths and saddles connecting $K$ to $K'.$
Genus zero implies that the core structure of this sequence is a tree of saddles, births and deaths. The genus zero surface is constructed from a sequence of  pairings of births with saddles, and saddles with deaths. In other words, the basic operation that constructs the concordance consists in the splitting off from,  or amalgamation of a trivial knot with the body of the concordance via a birth and saddle, or a saddle and a death.
Thus we can consider an elementary genus zero concordance consisting in a virtual knot $K$ and a trivial circle $C,$ disjoint from $K,$ such that the link diagram  $L$ consisting of the disjoint union of $K$ and $C$ undergoes virtual isotopy to a diagram $D.$ One oriented saddle point move on $D$  forms a new knot $K'$. It is sufficient to prove that $P_{K} = P_{K'}.$ To prove this fact, note that by taking a constant labeling of $C,$ we have a defined polynomial $P_{L}$ with $P_{K} = P_{L}.$ Then $L$ is isotopic to $D,$ and so by invariance of the affine index polynomial, $P_{K} = P_{L} = P_{D}.$ At the place of the saddle point move there is a label $a$ on the $K$ component of $D$ and a label $b$ on the $C$ component of $D.$ We can add $a-b$ to the labels on all arcs of the $C$ component of $D$ and retain a legal coloring of $D$ that does not change its polynomial evaluation (This is a general property of the labelings - they can always be shifted by a constant). Thus we may assume that $D$ is prepared with a labeling so that $P_{K} = P_{D},$ and the labels at the saddle point are the same. Then the saddle move can be performed, and the new diagram $K'$ inherits the same labeling. Hence $P_{D} = P_{K'}.$ We have proved that $P_{K} = P_{K'}.$  This completes the proof of the case of 
a birth followed by a saddle point. The remaining case is a saddle point followed by a death. In this case the link obtained after the saddle point inherits a labeling from the original knot and, given that the resulting link is isotopic to a disjoint union of a knot and a trivial circle, the argument proceeds as before. For links the criterion for the invariant to be defined is the existence of a labeling for the link diagram. Once we know that the labeling exists, the above arguments apply equally well to the case of links.\\

To complete the proof, we note that an elementary genus zero concordance from a link $L$ of two components $K$ and $K '$ with one saddle point as shown in Figure~\ref{onesaddle} has the property that  $P_{L} = P_{K} + P_{K'} = 0.$ The proof is by a labeling amalgamation argument as above. Similarly, if a concordance from knots 
$K$ to $K'$ consists in two saddle points as shown in Figure~\ref{twosaddle}, then $P_{K} = P_{K'}$ by two applications of the one saddle point observation.  These two types of saddle point interaction combined with the maximum and minimum cancellations with saddle poiints discussed above constitute a complete list of the possibilities in an arbitrary concordance.
See Figure~\ref{concordschema} for a typical example of a concordance schema. One sees, using the facts we have indicated here, that on passing through a critical level in the concordance, the value of the polynomial sum of the components of the link at that level is not changed. Thus the value of $P$ at the beginning of the concordance and the value of $P$ at the end are equal. This completes the proof that the affine index polynomial is 
an invariant of concordance of virtual knots and links.\\
\end{proof}
 
 \begin{figure}
     \begin{center}
     \includegraphics[width=5cm]{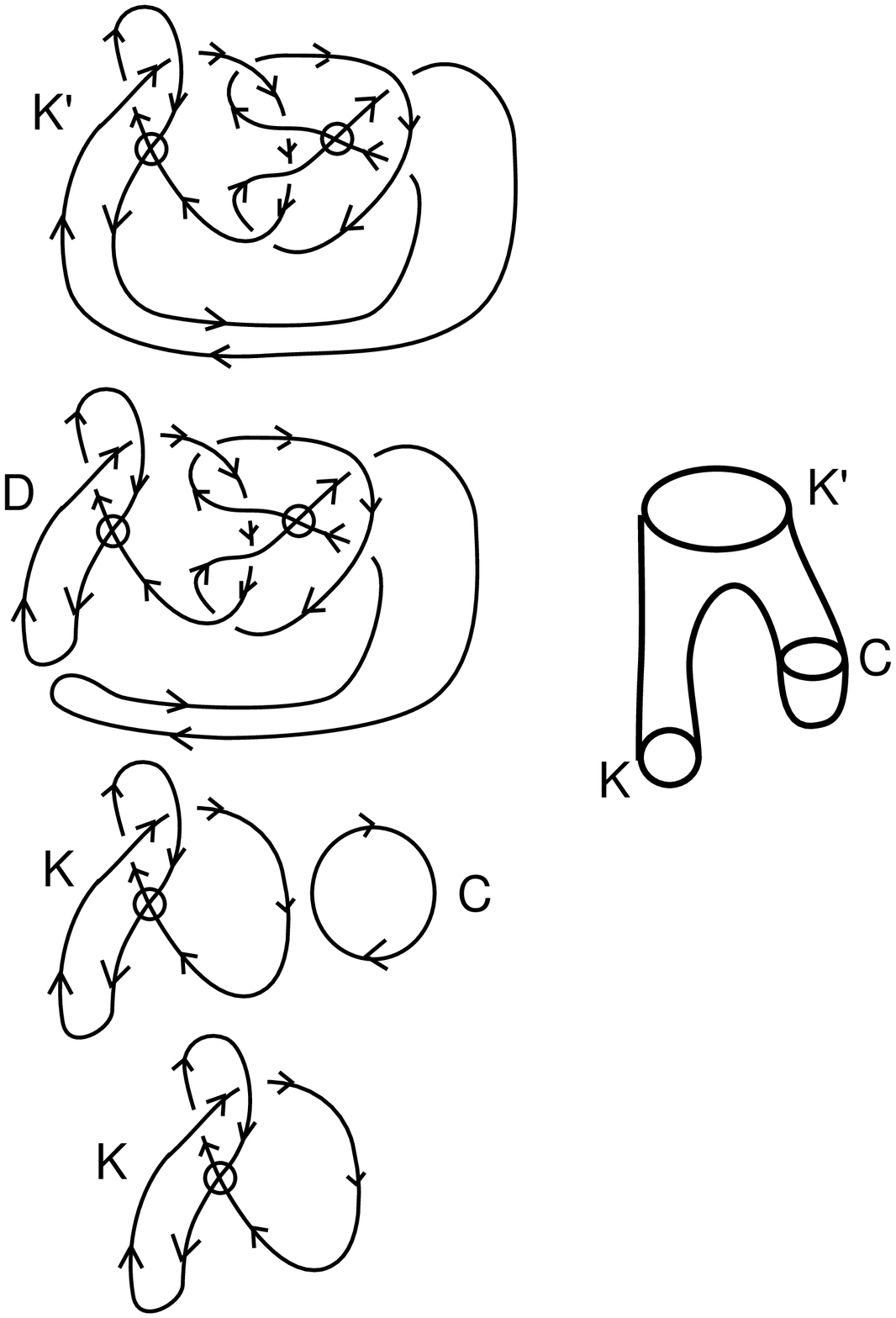}
     \caption{\bf An Elementary Concordance between $K$ and $K'$}
     \label{elem}
\end{center}
\end{figure}

\begin{figure}
     \begin{center}
     \includegraphics[width=5cm]{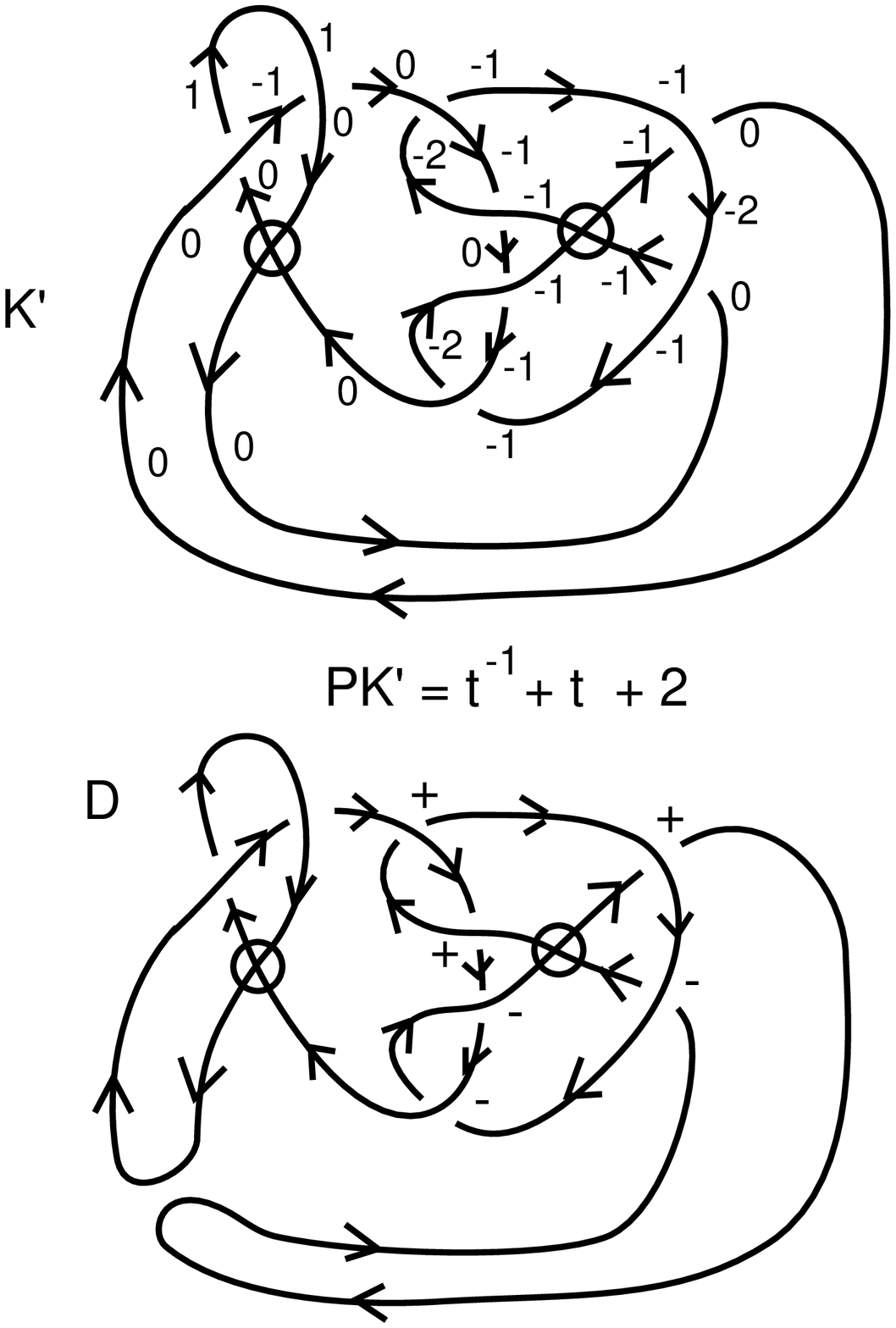}
     \caption{\bf An Elementary Labeled Concordance}
     \label{elemlabel}
\end{center}
\end{figure}

\begin{figure}
     \begin{center}
     \includegraphics[width=5cm]{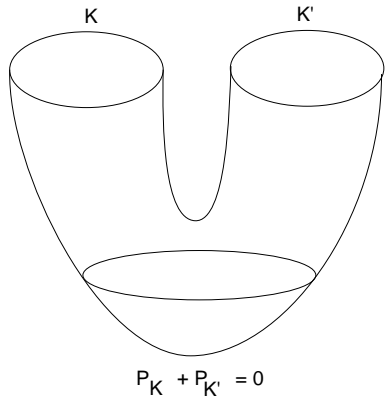}
     \caption{\bf Single Saddle Genus One Surface}
     \label{onesaddle}
\end{center}
\end{figure}

\begin{figure}
     \begin{center}
     \includegraphics[width=5cm]{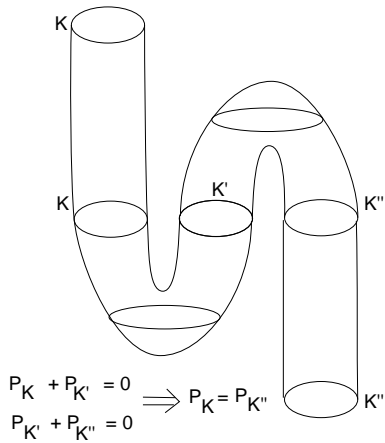}
     \caption{\bf Double Saddle Genus Zero Concordance}
     \label{twosaddle}
\end{center}
\end{figure}

\begin{figure}
     \begin{center}
     \includegraphics[width=5cm]{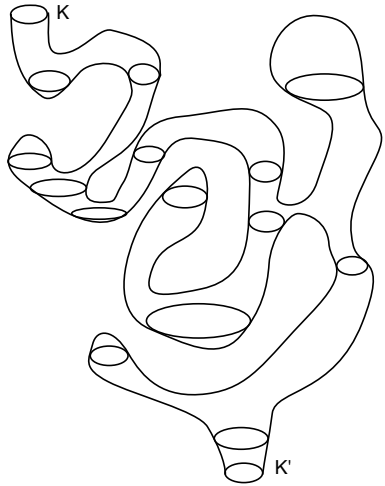}
     \caption{\bf Concordance Schema}
     \label{concordschema}
\end{center}
\end{figure}

\begin{figure}
     \begin{center}
     \includegraphics[width=5cm]{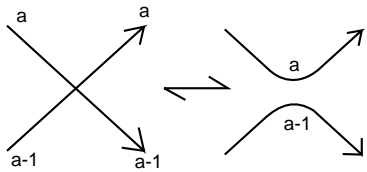}
     \caption{\bf Basic Labeled Cobordism}
     \label{basicob}
\end{center}
\end{figure}

\begin{figure}
     \begin{center}
     \includegraphics[width=6cm]{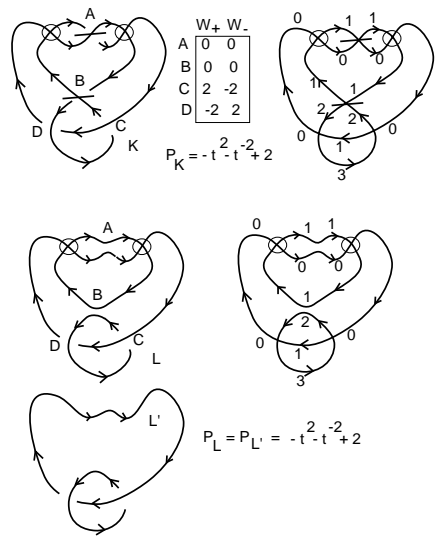}
     \caption{\bf Labeled Cobordism of a Knot to a Link}
     \label{cob}
\end{center}
\end{figure}

 \begin{figure}
     \begin{center}
     \includegraphics[width=6cm]{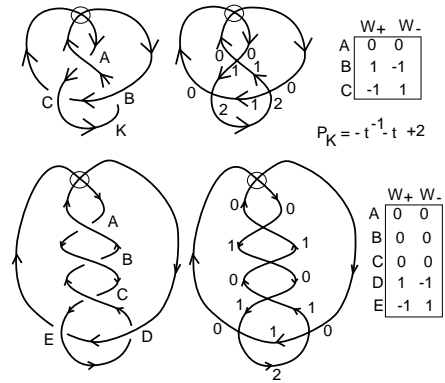}
     \caption{\bf A Family of Virtual Knots with the Same Polynomial}
     \label{family}
\end{center}
\end{figure}

\begin{figure}
     \begin{center}
     \begin{tabular}{c}
     \includegraphics[width=6cm]{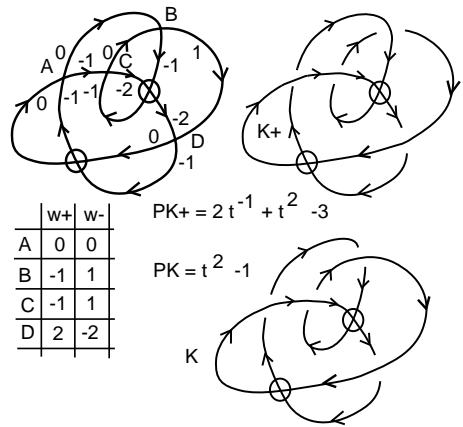}
     \end{tabular}
     \caption{\bf Polynomial Calculation for Two Knots}
     \label{two}
\end{center}
\end{figure}

\begin{figure}
     \begin{center}
     \begin{tabular}{c}
     \includegraphics[width=6cm]{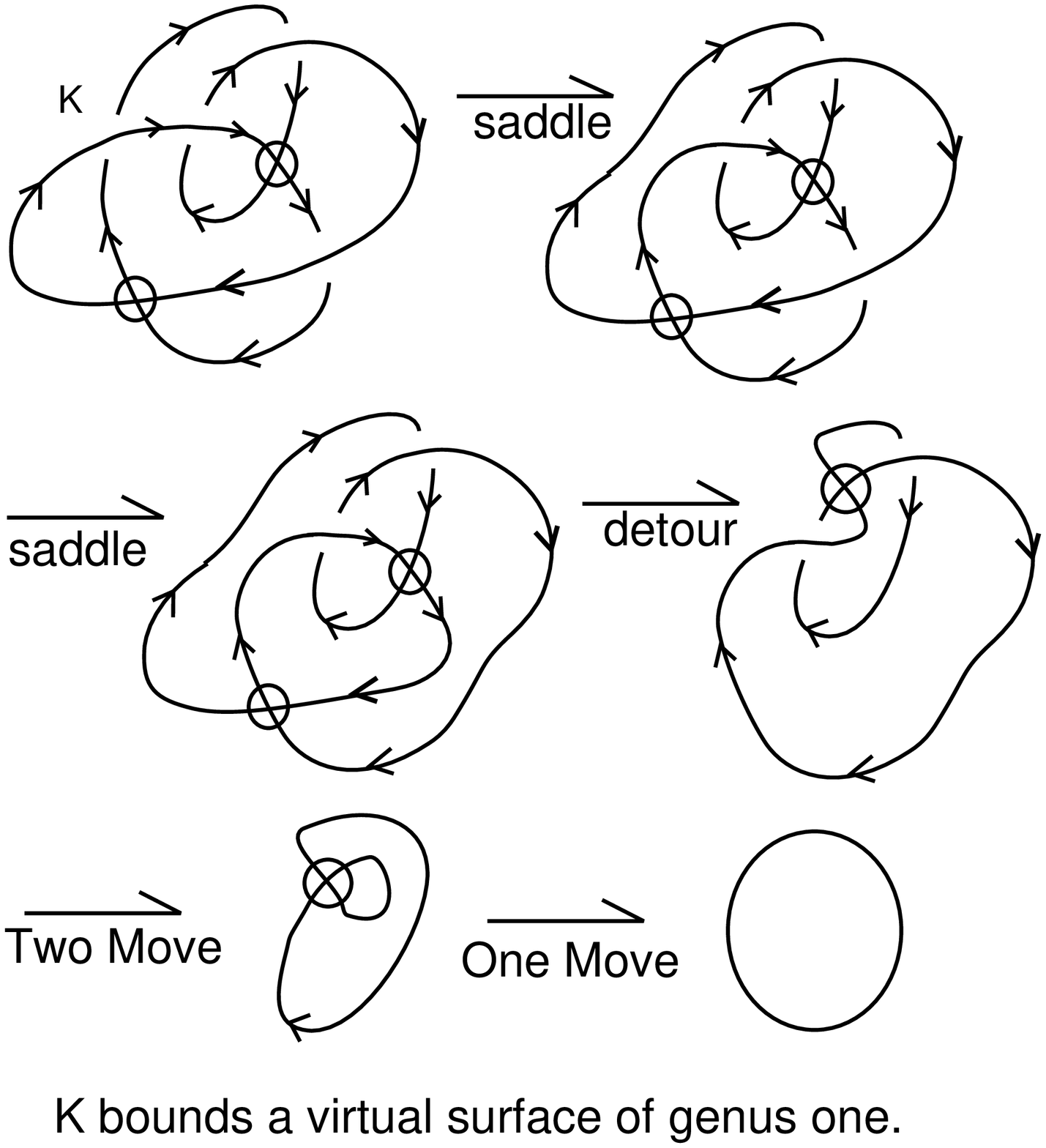}
     \end{tabular}
     \caption{\bf The Knot $K$ has virtual genus one.}
     \label{twotwo}
\end{center}
\end{figure}

\noindent {\bf Remark.} In Figure~\ref{elem} we illustrate an elementary concordance, as discussed in the proof above. The diagram $K'$ is transformed by a single saddle point move to the diagram $D$, which is isotopic
to a diagram that is the disjoint union of $K$ and $C$ where $C$ is an unknotted circle. Letting $C$ undergo death, we have a concordance from $K'$ to $K.$ We leave labeling this figure to the reader. It is clear that the 
crossings of the component of $D$ that becomes $C$ in the isotopy will have a total contribution of zero to the polynomial and that their contribution to $D$ is identical to their contribution to $K'.$ Thus we see directly in 
this case how $P_{K} = P_{K'}.$ In Figure~\ref{elemlabel} we show the weight calculation for the first part of the concordance in the previous figure. Note that the total weight contribution to the affine index polynomial from the unkotted and unlinked component (after the saddle move) is zero. This is in accord with the proof of Theorem 4.9.\\

\noindent {\bf Remark.} Any virtual slice knot $K$ will have $P_{K}(t) = 0$ since $K$ is concordant to the unknot. In the case of the virtual stevedore knot, we see in 
Figure~\ref{labelvstevedore} that all the weights are zero. We can ask when a virtual knot will have all of its weights equal to zero. It is certainly not the case that any virtual slice knot will have null weights.
For example, view Figure~\ref{consum} where we show the knot $K \sharp K^{!}$ where $K$ is the virtual trefoil, and $K^{!}$ denotes the vertical mirror image of $K.$ We know that $P_{K^{!}}(t) = -P_{K}(t)$ for any virtual knot $K.$ And so $P_{K \sharp K^{!}} = 0$ for any virtual knot $K.$ In fact, as remarked in the previous section, it is the case that $K \sharp K^{!}$ is virtually slice for any virtual knot $K.$ In such examples it is often the case that $P_{K}$ is non-trivial and so the diagram has canceling but non-null weights. This is the case in this specific example, where $P_{K} = t^{-1} + t - 2.$  \\

We finish this paper with a process that applies to most examples of the affine index polynomial. Taking a knot or link diagram $K$ with a labeling, some of the weights may be zero.
At each crossing with weight zero, we can smooth the crossing to obtain a link $L$ that is cobordant to $K$ (recall that smoothing a crossing can be accomplished by one saddle move). Thus we can smooth all crossings with null weights and obtain a knot of link $K'$ such that $K$ is cobordant to $K'$, $K'$ has only non-zero weights (or it is an unknot or unlink) and $P_{K} = P_{K'}.$ This process of removing crossings and making a cobordism that does not change the polynomial is particularly interesting in many examples. The link $K'$ in its way, contains the core of the invariant for $K$ and the remaining obstruction to making a concordance.
Here are descriptions of some examples of this phenomenon.\\

In Figure~\ref{basicob} and 
 Figure~\ref{cob} we illustrate how the appearance of zeroes in the list of vertex weights for the polynomial can be used to produce labeled knots and links where the crossings with null weights have been smoothed.  We will call the smoothing indicated in  Figure~\ref{basicob} a {\it basic labeled cobordism}. Thus if a knot has crossings with null weights, then it is labeled cobordant to a link with only non-zero weights (or an empty set of weights). While not all links can be labeled, this form of cobordism does produce labeled links, and the Index Invariant can be extended to such links as indicated in 
 Figure~\ref{hlink}. Here we write down the most general labeling for the link, and then deduce a set of variable integer exponents for the polynomial invariant, as described in Section 4.1.\\
 
 Figure~\ref{family} illustrates an infinite family of virtual knots with the same Affine Index Polynomial. Note that all of them are labeled cobordant to the Hopf link diagram. They can all be distinguished from one another
 by the bracket polynomial.\\
 
 In Figure~\ref{two} we illustrate the calculation of the affine index polynomial for two knots $K_{+}$ and $K.$ The knot $K_{+}$ is positive and by our theorem on the genus of positive virtual knots, it has genus two.
The knot $K$ is obtained from $K_{+}$ by switching one crossing. The affine index polynomial shows that it is not slice, and Figure~\ref{twotwo} shows that $K$ bounds a genus one virtual surface. Thus we know, using the affine index polynomial, that $K$ has genus equal to one.\\

\noindent {\bf Acknowledgements}

\noindent This work was supported by the Laboratory of Topology and Dynamics, Novosibirsk State University (contract no. 14.Y26.31.0025 with the Ministry of Education and Science of the Russian Federation).

\end{document}